\documentclass[12pt]{amsart}
\usepackage{amssymb,a4wide}
\usepackage[usenames]{color} 
\usepackage[T1]{fontenc}

\newcommand{\C}{{\mathbb C}}
\newcommand{\cf}{{\mathcal{F}}}

\newcommand{\la}{{\lambda}}
\newcommand{\z}{{\zeta}}

\renewcommand\Re{\operatorname{Re}}

\newtheorem{Theorem}{Theorem}
\newtheorem{Proposition}{Proposition}

\newtheorem{Remark}{Remark}
\title[The Spectrum of Volterra-type integration operators on Fock spaces]{The Spectrum of Volterra-type integration operators on generalized Fock spaces}

\author{Olivia Constantin and Anna-Maria Persson}
\date{}

\thanks{The first author was supported in part by the FWF project
P 24986-N25.}

\address{ Olivia Constantin,
Faculty of Mathematics,
University of Vienna,
Oskar-Morgenstern-Platz 1, 
1090 Vienna, 
Austria} \email {olivia.constantin@univie.ac.at}

\address{ Anna-Maria Persson, 
Department of Mathematics, Faculty of Science,
Lund University,
P.O. Box 118, S-221 00 Lund, Sweden} \email {Anna-Maria.Persson@math.lu.se}

\begin{document}
\maketitle

\begin{abstract} 
We describe the spectrum of certain integration operators acting on generalized  Fock spaces. 
\end{abstract}

{\small


}

\section{Introduction}
For analytic functions $f,g$ we consider the Volterra-type integration operator given by
$$
T_g f(z)=\int_0^z fg'\, d\z.
$$
There is an extensive literature regarding the boundedness, compactness, and Schatten class membership of 
$T_g$ on various spaces of analytic functions (see \cite{cima,asis} for Hardy spaces, \cite{ac,asi,oc,pp,ra,pelra} for weighted Bergman spaces,
\cite{ggp,gp} for Dirichlet spaces,  \cite{oc_f,ocpel} for Fock spaces, \cite{bonet,bonet-taskinen} for growth spaces of entire functions, 
as well as the surveys \cite{aa,sis} and references therein). One of the key tools in all these considerations is a Littlewood-Paley-type 
estimate for the target space of the operator, which, for a wide class of radial weights, can be obtained by standard techniques. In 
comparison to the above, there are considerably fewer investigations of its spectrum (see \cite{ac,aleman-pelaez,bonet}). 
The problem of describing the spectrum of $T_g$ is in general quite involved as it requires Littlewood-Paley-type 
estimates for non-radial weights dependent on the symbol $g$.

Assuming that $g(0)=0$, a straightforward calculation shows that for $\la\in\C\setminus \{0\}$, the equation 
$$
f-\frac{1}{\la} T_g f=h
$$ 
has the unique analytic solution
\begin{equation}\label{resolvent}
f(z)=R_{\la,g} h(z)= h(0) e^{g(z)/\la}+ e^{g(z)/\la}\int_0^z e^{-g(\z)/\la} h'(\z) d\z,\quad z\in\C.
\end{equation}
Thus the resolvent set of $T_g$ consists precisely of those points $\la\in\C$ for which $R_{\la,g}$
is a bounded operator. One useful feature is that the spectrum of 
$T_g$ is closely related to the behaviour of the exponentials $e^{g/\lambda}$, $\la\in\C\setminus \{0\}$. 
Notice that for $h\equiv 1$ in (\ref{resolvent}) we get $ R_{\la,g} 1= e^{g/\la}$. This shows that if $T_g$ is bounded 
on some Banach space $X$ of analytic functions which contains the constants and on which point evaluations are bounded, then 
$e^{g/\la}\in X$ whenever $\la$ belongs to the resolvent set of $T_g$. This fact was first noticed 
by Pommerenke \cite{pommerenke} in the setting of $X=H^2$, who relied on it in order to provide a slick proof of the John-Nirenberg inequality. 
Hence, if $X$ is as above, then
$$\{0\} \cup \overline{\{\la \in \C \setminus\{0\}:\ e^{g/\la} \not\in X\}}\subset \sigma(T_g|X)\,.$$
Whether the reverse inclusion holds for all symbols $g$ for which $T_g$ is bounded depends on the space $X$. For example, 
this fails in the case of Bergman spaces and  of Hardy spaces (see \cite{ac,aleman-pelaez}).

Another observation worthwhile making is that the point spectrum of $T_g$ is empty. This is an immediate consequence of the 
form of $T_g$. In particular, if $T_g$ is compact, then $\sigma(T_g)=\{0\}$.  

The aim of this note is to describe the spectrum of $T_g$ when acting on the generalized Fock spaces $\cf^p_{\alpha,A}$ with $\ p\ge1,\, \alpha,A>0$,
which consist of entire functions $f$ such that
$$
\|f\|_{p,\alpha,A}=
\Bigl( 
\int_\C \Bigl| f(z) e^{-\alpha |z|^A}\Bigr|^p dA(z)
\Bigr)^{\frac{1}{p}}<\infty,
$$  
where $dA$ denotes the Lebesgue area measure on $\C$. 
The operator $T_g$ is bounded on $\cf^p_{\alpha,A}$ if and only if the symbol $g$ is a polynomial of degree $\leq A$, 
while its compactness is equivalent to $degree(g)< A$ (see \cite{ocpel}).
For integer values of A we obtain the following description for the spectrum of $T_g:\cf^p_{\alpha,A}\rightarrow \cf^p_{\alpha,A}$
$$
 \sigma(T_g)=\{0\}\cup\overline{\{\la \in \C \setminus\{0\}:\ e^{g/\la} \not\in  \cf^p_{\alpha,A} \}}= \Bigl \{\la \in \C :\ |\la|\le \frac{|b|}{\alpha}\Bigr\}, 
$$ 
where $b$ is the leading coefficient of the polynomial $g$ (that is, $g(z)=bz^{A}+"lower\  terms"$). If $A\not\in \mathbb{Z}$, it turns out that
the operator $T_g$ is automatically compact once it is bounded, and, since its point spectrum is empty, we obtain $\sigma(T_g)=\{0\}$.
The particular case $p=2, A=\frac{1}{2}$, was treated in \cite{oc_f} by a method based on the Hilbert space property of the classical
Fock space, as well as the explicit formula for its reproducing kernel. Our present approach relies on a suitable application of Stokes' formula,
using some properties of the integral means of functions in $\cf^p_{\alpha,A}$.
The case $p=\infty$, corresponding to growth spaces of entire functions $\cf^\infty_{\alpha,A}$,  was recently treated by Bonet in \cite{bonet}.
\bigskip\bigskip

\section{Main result}
\noindent Let us begin with some preliminary facts that will be used in our further considerations.
The following Littlewood-Paley estimate for $\cf^p_{\alpha,A}$ was obtained in \cite{ocpel}.
\medskip

\noindent {\bf Proposition A.} {\it We have
%The norm on $\cf^p_{\varphi}$ is comparable to 
$$\|f\|_{p,\alpha,A}^p\sim|f(0)|^p+\int_{\C} |f'(z)|^p\frac{e^{-p|z|^A}}{(1+|z|)^{p(A-1)}}dA(z),$$
\noindent for any entire function $f$.}

\begin{Remark}\label{remarca}
For an entire function $f$, the maximum principle and the subharmonicity
of $|f|^p$ ensure
\begin{equation*}\label{maxprinc}
\int_\C\frac{|f(z)|^p}{(1+|z|)^p} e^{-p\alpha |z|^A} dA(z) \sim \int_{|z|>1} \frac{|f(z)|^p}{(1+|z|)^p} e^{-p\alpha |z|^A} dA(z)\sim
\int_{|z|>1} \Bigl|\frac{f(z)}{z}\Bigr|^p e^{-p\alpha |z|^A} dA(z).
\end{equation*}
Moreover, if $f(0)=0$, the above quantities are comparable to $\int_{\C} \left|\frac{f(z)}{z}\right|^p e^{-p\alpha |z|^A} dA(z)$.
\end{Remark}

The key ingredient in the proof of our main result is the following estimate:

\begin{Proposition}\label{LPF}
Let $b,\la\in\C$ and assume $\alpha>|b/\la|$.  For a positive integer $A$, let $g(z)=bz^A,\, z\in\C$.  If $f$ is an entire function such that
$fe^{g/\la}\in \cf^p_{\alpha,A}$, then the following inequality holds
\begin{equation}\label{desired_ineq}
\int_{\C}|f(z)|^p |e^\frac{g(z)}{\la}|^pe^{-p\alpha|z|^A}dA(z)\lesssim|f(0)|^p+\int_{\C}|f'(z)|^p |e^{\frac{g(z)}{\la}}|^p\frac{e^{-p\alpha|z|^A}}{(1+|z|)^{(A-1)p}}dA(z),
\end{equation}
where the involved constants are independent of $f$.

%Let $g(z)=bz^A$, $b\in\C$, $A\in\mathbb{N}$, $A\geq 2$ and consider $\displaystyle w(z)=p\Re \left(\frac{g(z)}{\la}\right)-p\alpha |z|^A$ with $\alpha>\left|\frac{b}{\la}\right|$.  If $f$ has a zero of order strictly greater than $A$ at zero then
%\begin{equation}\label{eq:desired_ineq}
%\int_{\C}|f(z)|^pe^{w(z)}dA(z)\lesssim\int_{\C}\frac{|f'(z)|^p}{|z|^{(A-1)p}}\cdot e^{w(z)}dA(z).
%\end{equation}
\end{Proposition}

\begin{proof}
We shall first prove the statement for functions $f$ that have a zero of order strictly greater than $A$ at zero.

Denote $\displaystyle w(z)=p\Re \left(\frac{g(z)}{\la}\right)-p\alpha |z|^A,\, z\in\C$. 
Applying Stokes' Theorem and a limiting argument we deduce
\begin{align}\label{stokes}
\int_{\C}|f|^pe^wdA=&\int_{\C}|f|^p\frac{\bar{\partial} (e^w)}{\bar{\partial}w}dA=-\int_{\C}\bar{\partial}\left(\frac{|f|^p}{\bar{\partial} w}\right)e^wdA\\\nonumber
=&-\int_{\C}\bar{\partial}\left(|f|^p\right)\frac{1}{\bar{\partial} w}e^wdA+\int_{\C}|f|^p\frac{\bar{\partial} ^2w}{(\bar{\partial}w)^2}e^wdA .\nonumber
\end{align}
Above we used the fact that  
\begin{equation}\label{eq:boundaryterm}
\lim_{R\to\infty}\Bigr|\int_{|z|=R}\frac{|f|^p}{\bar{\partial}w}\cdot e^wdz\bigl|=0,
\end{equation}
whose proof is deferred for later for the sake of clarity of exposition.
To estimate the first integral on the right hand side of (\ref{stokes}), denoted $I_1$,  we write
\begin{align}\label{I1}
|I_1|=\left|\int_{\C}\bar{\partial}\left(|f|^p\right)\frac{1}{\bar{\partial} w}e^wdA\right|=&\frac{p}{2}\left|\int_{\C}\frac{|f|^p}{\bar f}\bar{f'}\cdot\frac{e^w}{\bar{\partial}w}dA\right|\\\nonumber
\lesssim&\int_{\C}|f|^{p-1}\left|\frac{f'}{\bar{\partial}w}\right|e^wdA\\\nonumber
\leq&\left(\int_{\C}|f|^pe^wdA\right)^{\frac{p-1}{p}}\cdot\left(\int_{\C}\frac{|f'|^p}{|\bar{\partial}w|^p}e^wdA\right)^{\frac{1}{p}},\nonumber
\end{align}
where the last step above follows by H\"older's inequality.

\noindent A direct calculation now gives
\begin{align*}
\bar{\partial}w
%=&\bar{\partial}\left(p\left(\Re\left(\frac{b}{\la}z^A\right)-\alpha|z|^A\right)\right)\\\nonumber
%=&\bar{\partial}\left(p\left(\frac{b}{2\la}z^{A}+\frac{\bar{b}}{2\bar{\la}}\bar{z}^A-\alpha z^{\frac{A}{2}}\bar{z}^{\frac{A}{2}}\right)\right)\\\nonumber 
=&p\left(\frac{\bar{b}A}{2\bar{\la}}\bar{z}^{A-1}-\frac{\alpha A}{2} z^{\frac{A}{2}}\bar{z}^{\frac{A}{2}-1}\right),\\\nonumber 
\bar{\partial}^2w=&p\left(\frac{\bar{b}A(A-1)}{2\bar{\la}}\bar{z}^{A-2}-\frac{\alpha A(A/2-1)}{2} z^{\frac{A}{2}}\bar{z}^{\frac{A}{2}-2}\right),\nonumber
\end{align*}
and hence we have
\begin{equation}\label{eq:star}
|\bar{\partial}w|\geq\frac{pA}{2}\left(-\left|\frac{b}{\la}\right|+\alpha\right)|z|^{A-1}, \quad\quad |\bar{\partial}^2w|\lesssim|z|^{A-2}.
\end{equation}
%By the triangle inequality we get 
%\begin{equation}\label{eq:star}
%|\bar{\partial}w|\geq\frac{pA}{2}\left(-\left|\frac{b}{\la}\right|+\alpha\right)|z|^{A-1}.
%\end{equation}
%On the other hand,
%$$\bar{\partial}^2w=p\left(\frac{\bar{b}A(A-1)}{2\bar{\la}}\bar{z}^{A-2}-\frac{\alpha A(A/2-1)}{2} z^{\frac{A}{2}}\bar{z}^{\frac{A}{2}-2}\right)$$
%and hence
%\begin{equation}\label{eq:starstar}
%|\bar{\partial}^2w|\lesssim|z|^{A-2}.
%\end{equation}
Use this in (\ref{I1}) to get
$$|I_1|\lesssim \left(\int_{\C}|f|^pe^wdA\right)^{\frac{p-1}{p}}\cdot\left(\int_{\C}\frac{|f'|^p}{|z|^{(A-1)p}}e^wdA\right)^{\frac{1}{p}},$$
which together with relation (\ref{stokes}) yields
\begin{align}\label{eq:assumed_ineq}
\int_{\C}|f|^pe^wdA\lesssim& \left(\int_{\C}|f|^pe^wdA\right)^{\frac{p-1}{p}}\cdot\left(\int_{\C}\frac{|f'|^p}{|z|^{(A-1)p}}e^wdA\right)^{\frac{1}{p}}+\int_{\C}|f|^p\frac{|z|^{A-2}}{|z|^{2A-2}}e^wdA \\\nonumber
=& \left(\int_{\C}|f|^pe^wdA\right)^{\frac{p-1}{p}}\cdot\left(\int_{\C}\frac{|f'|^p}{|z|^{(A-1)p}}e^wdA\right)^{\frac{1}{p}}+\int_{\C}\frac{|f|^p}{|z|^{A}}e^wdA .\nonumber 
\end{align}
%Combining this inequality with an argument by contradiction we can now show  (\ref{eq:desired_ineq}), i.e.
We now claim that the above relation implies
\begin{equation}\label{claim1}
 \left(\int_{\C}|f|^pe^wdA\right)^{\frac{1}{p}}\lesssim\left(\int_{\C}\frac{|f'|^p}{|z|^{(A-1)p}}e^wdA\right)^{\frac{1}{p}},
\end{equation}
if $f$ has a zero of order $>A$ at the origin. Once (\ref{claim1}) is proven,  (\ref{desired_ineq}) will follow in view of Remark \ref{remarca}.
We use an argument by contradiction.
If (\ref{claim1}) did not hold, then we could find a sequence $(f_n)_{n=1}^{\infty}$ of entire functions with $\|f_n\|_{L^p(e^w)}=1$ and $\displaystyle\int_{\C}\frac{|f_n'|^p}{|z|^{(A-1)p}}e^wdA<\frac{1}{n}$.
For any compact $K\subset\C$, the point evaluation estimate for $\cf^p_{\alpha,A}$ (see \cite{ortega,ocpel}) gives
$$\left|\frac{f_n'(z)}{z^{A-1}}\right|\cdot|e^{\frac{g(z)}{\la}}|\lesssim  c_K \Bigl\| \frac{f_n'(z)}{z^{A-1}}e^{\frac{g(z)}{\la}}\Bigr\|_{p, \alpha,A}\le c_K\left(\frac{1}{n}\right)^{\frac{1}{p}},$$
for some constant $c_K>0$.
%e^{\alpha|z|^A}\cdot(\Delta|z|^A)^{\frac{1}{p}}\cdot\left(\frac{1}{n}\right)^{\frac{1}{p}}.$$
%On each compact set $K\subset\C$ we have $\displaystyle |e^{\frac{g}{\la}}|\sim 1$,  $\displaystyle e^{\alpha|z|^A}\sim 1$ (and $\Delta|z|^A\sim|z|^{A-2}$)
%which implies that
%$$|f_n'(z)|\lesssim c_K\left(\frac{1}{n}\right)^{\frac{1}{p}}$$ for some constant $c_K$ depending on $K$. 
This shows that  the sequence $f_n'$ converges to $0$ uniformly on compacts, and hence the same holds for the sequence $f_n$.

From (\ref{eq:assumed_ineq}) we deduce
$$1\lesssim\left(\frac{1}{n}\right)^{\frac{1}{p}}+\int_{|z|\leq R}\frac{|f_n|^p}{|z|^A}e^wdA+\frac{1}{R^A}\int_{|z|>R}|f_n|^pe^wdA.$$
Since by the maximum principle $\displaystyle \frac{f_n}{z^A}$ tends uniformly to $0$ on $|z|\leq R$, the first integral on the right-hand side of the inequality above tends to $0$ as $n\to \infty$. We now use the fact that $\|f_n\|_{L^p(e^w)}=1$  and let $R\rightarrow\infty$ 
%would imply that $1\lesssim \frac{1}{R^A}$ for any $R$ which gives 
to get the desired contradiction.  

It remains to show that the condition (\ref{eq:boundaryterm}) involving the boundary integral holds.
By (\ref{eq:star}) we obtain
\begin{align}\label{eq:boundary_term}
\Bigr|\int_{|z|=R}\frac{|f|^p}{\bar{\partial}w}\cdot e^wdz\bigl|
%=&\int_{|z|=R}\frac{|fe^{\frac{g}{\la}}|^pe^{-p\alpha|z|^A}}{|\bar{\partial}w|}d|z|\nonumber\\
\lesssim&\int_{0}^{2\pi}\frac{|(fe^{\frac{g}{\la}})(Re^{i\theta})|^pe^{-p\alpha R^A}}{|\bar{\partial}w(Re^{i\theta})|}Rd\theta\nonumber\\
\sim\ &\frac{e^{-p\alpha R^A}}{R^{A-2}}\int_{0}^{2\pi} |(fe^{\frac{g}{\la}})(Re^{i\theta})|^p d\theta.
\end{align}
Let now $\psi=fe^{\frac{g}{\la}}$ and consider $\displaystyle M_{p,R}^p(\psi)=\int_0^{2\pi}|\psi(Re^{i\theta})|^pd\theta$.
%For $\psi\in\cf^p_{\alpha,A}$ we have  
By assumption, $\psi\in\cf^p_{\alpha,A}$, that is,
$$\|\psi\|_{\cf^p_{\alpha,A}}=\int_0^\infty M_{p,r}^p(\psi) e^{-p\alpha r^A}rdr<\infty$$
and hence
$$\lim_{R\to\infty}\int_R^\infty r e^{-p\alpha r^A} M_{p,r}^p(\psi) dr =0.$$ 
Since the integral means  $M_{p,R}$ are increasing in $R$, this implies that 
$$ \lim_{R\to\infty}M_{p,R}^p(\psi)\int_R^\infty r e^{-p\alpha r^A} dr =0,$$
and therefore the right-hand side in (\ref{eq:boundary_term}) can be written as 
\begin{align}\label{eq:boundary_term_again}
\frac{e^{-p\alpha R^A}}{R^{A-2}}M_{p,R}^p(\psi)=&\frac{e^{-p\alpha R^A}\cdot R^{2-A}}{\int_R^\infty r e^{-p\alpha r^{A}}dr}M_{p,R}^p(\psi)\int_R^\infty r e^{-p\alpha r^{A}}dr\nonumber\\
=&\frac{e^{-p\alpha R^A}\cdot R^{2-A}}{\int_R^\infty r e^{-p\alpha r^{A}}dr} o(1)\quad \hbox{ as } R\rightarrow\infty.
\nonumber
\end{align}
By L'Hospital's rule we see that the first factor above is bounded and hence (\ref{eq:boundaryterm}) is proven. 

%With this, the desired inequality ( \ref{desired_ineq}) is proven for functions $f$ with a zero of order greater than $A$ at the origin. 

\noindent The conclusion for arbitrary $f$ is obtained by applying (\ref{eq:boundaryterm}) to $\displaystyle f(z)-\sum_{k=0}^{A-1}\frac{f^{(k)}(0)}{k!}z^k$, taking into account Remark \ref{remarca}, and the fact that 
$$|f^{(k)}(0)|\lesssim\int_{\C}\frac{|f'|^p}{(|z|+1)^{(A-1)p}}e^w dA,\quad 1\leq k\leq A,$$
%as well as the fact that 
%$$\int_{\C}\frac{|f|^p}{(1+|z|)^{pB}}e^{-\alpha p|z|^A}\sim \int_{|z|>1}\frac{|f|^p}{(1+|z|)^{pB}}e^{-\alpha p|z|^A}\sim \int_{|z|>1}\left|\frac{f(z)}{z^B}\right|e^{-\alpha p|z|^A},$$
which follows by the Cauchy formula and subharmonicity.
\end{proof}

\begin{Theorem} Let $\alpha, A>0,\, p\ge 1$ and assume  $T_g: \cf^p_{\alpha,A}\rightarrow\cf^p_{\alpha,A}$ is bounded, that is, $g$ is a polynomial with $\hbox{degree}(g)\le A$.
Then we have
\smallskip

\begin{enumerate}
\item[$\hbox{(I)}$] If $A\not\in\mathbb{N}$, then $\sigma(T_g)=\{0\}$.
\medskip

\item[$\hbox{(II)}$] If $A\in\mathbb{N}$, then 
\begin{equation}\label{spectru}
 \sigma(T_g)
 %=\{0\}\cup\{\la \in \C \setminus\{0\}:\ e^{g/\la} \not\in  \cf^p_{\alpha,A} \}
 = \Bigl \{\la \in \C :\ |\la|\le \frac{|b|}{\alpha}\Bigr\}
\end{equation}
where $b$ is the coefficient of $z^A$ in the expansion of $g$.
\end{enumerate}
\end{Theorem}
\begin{proof}
If $A$ is not a positive integer, then $T_g$
% is bounded on $\cf^p_{\varphi}=\cf^p_{\alpha|z|^A}$ if and only if $g$ is a polynomial of degree strictly less than $A$. This implies that $T_g$ is automatically 
compact $\cf^p_{\alpha,A}$ once it is bounded (see \cite{ocpel}), and since it has no eigenvalues, its spectrum is $\{0\}$. 

Thus, the only interesting case is $A\in\mathbb{N}$, i.e. $g$ is of the form $g(z)=bz^A+p_{A-1}(z)$ where $p_{A-1}$ is a polynomial of degree $A-1$.
Notice that for $\la\neq 0$ we have
$$I-\frac{1}{\la}T_g=(I-\frac{1}{\la}T_{bz^A}) -\frac{1}{\la}T_{p_{A-1}} ,$$
and since the operators $I-\frac{1}{\la}T_g$ and $I-\frac{1}{\la}T_{bz^A}$ are injective, while $T_{p_{A-1}} $ is compact on $\cf^p_{\alpha,A}$,
the study of the spectrum of $T_g$ reduces in this case to study of the spectrum of $T_{bz^A}$. 
We can therefore assume without loss of generality that $g(z)=bz^A$.
%We claim that $$\sigma (T_{bz^A})=\{\lambda\in\C\,:\,|\la|\leq\frac{|b|}{\alpha}\}.$$

Now it is easy to see that the spectrum of $T_g$ contains the set in the right-hand side of (\ref{spectru}). Indeed, as explained in Introduction, 
if $\la$  belongs to  the resolvent  set of $T_g $ then   $R_{\la,g} 1=e^{bz^A/\la}\in\cf^p_{\alpha, A}$, that is
$$\int_{\C}e^{p\Re \left(\frac{bz^A}{\la}\right)-p\alpha|z|^A} dA(z)<\infty.$$
A straightforward argument using polar coordinates shows that this fails if  $\left|\frac{b}{\la}\right|>\alpha$,
and hence $\overline{D(0,|b|/\alpha)}\subseteq \sigma(T_g)$.

\noindent To prove the reverse inclusion, we apply Proposition \ref{LPF}. We show that if $\left|\frac{b}{\la}\right|<\alpha$, 
 then the formal resolvent $R_{\la,g}$, given by (\ref{resolvent}), is a bounded operator on $\cf^p_{\alpha,A}$.
We claim that, for any polynomial $f$,  
$$
e^{g/\la}\int_0^z e^{-g(\z)/\la} f'(\z) d\z\in \cf^p_{\alpha,A}.
$$ 
Indeed, this follows from the estimate
$$
\left |e^{g/\la}\int_0^z e^{-g(\z)/\la} f'(\z) d\z\right|
\le\left| \int_0^1 e^{\frac{b}{\la}z^A(1-t^A)} f'(tz) z dt  \right|\le e^{\frac{|b|}{|\la|}|z|^A} \int_0^1 |zf'(tz)| dt,
$$
taking into account $|{b}/{\la}|<\alpha$.

We may now apply Proposition \ref{LPF} and, subsequently, Proposition A  to obtain
\begin{eqnarray*}
\|R_{\la,g} f\|_{p,\alpha,A}&\le& |f(0)| \|e^{g(z)/\la}\|_{p,\alpha,A}+ \Bigl\|e^{g(z)/\la}\int_0^z e^{-g(\z)/\la} f'(\z) d\z\Bigr\|_{p,\alpha,A}\\
&\lesssim& \|f\|_{p,\alpha,A}+ \Bigl(\int_\C \frac{|f'(z)|^p}{(1+|z|)^p} e^{-p\alpha |z|^A} dA(z) \Bigr)^{1/p}\lesssim \|f\|_{p,\alpha,A}.
\end{eqnarray*}
Since polynomials are dense in $\cf^p_{\alpha,A}$ (see e.g. \cite{ocpel}), it follows that $R_{\la,g}$ is bounded, and, with this,
 the proof is complete. 
\end{proof}


\medskip


\begin{thebibliography}{99}

%\bibitem{agmon}
%S. Agmon, {\it Sur un probl\'eme de translations},  C. R. Acad. Sci. Paris {bf 229}, (1949), 540--542.

\bibitem{aa}
A. Aleman, {\it A class of integral operators on spaces of analytic functions}, Topics in complex analysis and operator theory, 3--30, Univ. Malaga, Malaga, 2007. 

\bibitem{cima} A. Aleman, J. A. Cima, {\it An integral operator on $H^p$ and Hardy's inequality}, J. Anal. Math. {\bf 85} 
(2001), 157--176.

\bibitem{ac}
A. Aleman and O. Constantin, {\it Spectra of integration operators on weighted Bergman spaces}, J. Anal. Math. {\bf 109} (2009), 199--231.

\bibitem{aleman-pelaez}
A. Aleman and J. Pelaez, {\it Spectra of integration operators and weighted square functions}, Indiana Univ. Math. J. {\bf 61} (2012), no. 2, 775--793.

\bibitem{asi}
A. Aleman, A. G. Siskakis, {\it Integration operators on Bergman spaces}, Indiana Univ. Math. J. {\bf 46} (1997),  337--356. 

\bibitem{asis}
A. Aleman,  A. G. Siskakis, {\it An integral operator on $H^p$},  Complex Variables Theory Appl. {\bf 28} (1995), 149--158.

%\bibitem{gelfand} I. M. Gelfand, {\it A problem}, Uspehi Matem. Nauk {\bf 5} (1938), 223.
\bibitem{bonet} J. Bonet, {\it The spectrum of Volterra operators on weighted spaces of entire functions}, Quart. J. Math. doi:10.1093/qmath/hav019


\bibitem{bonet-taskinen} J. Bonet and J. Taskinen 
{\it A note about Volterra operators on weighted Banach spaces of entire functions}, Math. Nachr. doi:10.1002/mana.201400099


\bibitem{oc} O. Constantin, {\it Carleson embeddings and some classes of operators on weighted Bergman spaces}, J. Math. Anal. Appl. {\bf 365} (2010),  668--682.

\bibitem{oc_f} O. Constantin, {\em{A Volterra-type integration operator on Fock spaces}}, Proc. Amer. Math. Soc.
{\bf{140}} (2012), 4247--4257.

\bibitem{ocpel} O. Constantin and J.A. Pelaez, {\it Integral operators, embedding theorems and a Littlewood--Paley formula on weighted Fock spaces}, J. Geom. Anal.
doi:10.1007/s12220--015--9585--7

\bibitem{ggp} P. Galanopoulos, D. Girela, J. A. Pelaez, {\it Multipliers and integration operators on Dirichlet spaces}, Trans. Amer. Math. Soc. {\bf 363} (2011), 1855--1886.

\bibitem{gp}
D. Girela, J.A. Pelaez, {\it Carleson measures, multipliers and integration operators for spaces of Dirichlet type}, J. Funct. Anal. {\bf 241} (2006), 334--358. 

\bibitem{ortega} N. Marco, X. Massaneda, J. Ortega-Cerda, {\it Interpolating and sampling sequences for entire functions},
 Geom. Funct. Anal. {\bf 13} (2003), 862--914. 

\bibitem{pp}
J. Pau, J. A. Pelaez, {\it Embedding theorems and integration operators on Bergman spaces with rapidly decreasing weights}, J. Funct. Anal. {\bf 259} (2010), 2727--2756

\bibitem{pelra} J. A. Pelaez and J. R\"atty\"a,  {\it Weighted Bergman spaces induced by rapidly increasing weights}, Mem. Amer. Math. Soc. {\bf 227} (2014), no. 1066. 

\bibitem{pommerenke} Ch. Pommerenke,  {\it Schlichte Funktionen und analytische Funktionen von beschr\"ankter mittlerer Oszillation}, Comment. Math. Helv. {\bf 52} (1977), 591--602.  

\bibitem{ra}
J. R\"atty\"a,  {\it Integration operator acting on Hardy and weighted Bergman spaces}, Bull. Austral. Math. Soc. {\bf 75} (2007), 431--446.

\bibitem{sis}
A. G. Siskakis, {\it Volterra operators on spaces of analytic functions --- a survey}, Proceedings of the First Advanced Course in Operator Theory and Complex Analysis, 51--68, Univ. Sevilla Secr. Publ., Seville, 2006
\end{thebibliography}
\end{document}